\documentclass[a4paper]{amsart}

\usepackage{amssymb,amscd,colonequals,enumerate}
\usepackage{mathrsfs}
\usepackage[all]{xy}

\usepackage[
backref,
pdfauthor={Han Wu and Chang Lv},  
pdftitle={Genus one half stacky curves violating the local-global principle},
]{hyperref}
\usepackage[alphabetic,backrefs,lite]{amsrefs}  

\theoremstyle{plain}
\theoremstyle{definition}

\newtheorem{theorem}{Theorem}[subsection]
\newtheorem{thm}{Theorem}[subsubsection]
\newtheorem{lemma}[theorem]{Lemma}

\newtheorem{proposition}[theorem]{Proposition}

\theoremstyle{definition}

\theoremstyle{remark}
\newtheorem{remark}[theorem]{Remark}

\newcommand{\QQ}{\mathbb{Q}}

\newcommand{\ZZ}{\mathbb{Z}}

\newcommand{\FF}{\mathbb{F}}
\newcommand{\GG}{\mathbb{G}}
\newcommand{\PP}{\mathbb{P}}

\newcommand{\Ocal}{{\mathcal O}}

\newcommand{\calP}{\mathcal{P}}
\newcommand{\calX}{\mathcal{X}}
\newcommand{\calY}{\mathcal{Y}}

\DeclareMathOperator{\Proj}{Proj}

\DeclareMathOperator{\Spec}{Spec}

\DeclareMathOperator{\Pic}{Pic}

\usepackage[OT2,T1]{fontenc}
\DeclareSymbolFont{cyrletters}{OT2}{wncyr}{m}{n}
\DeclareMathSymbol{\Sha}{\mathalpha}{cyrletters}{"58}

\newcommand{\defi}[1]{\textsf{#1}} 

\newcommand{\HH}{{\operatorname{H}}}

\makeatletter
\g@addto@macro\bfseries{\boldmath}  
\makeatother

\begin{document}
	
	\begin{title}
		{Genus one half stacky curves violating the local-global principle}  
	\end{title}
	\author{Han Wu}
	\address{University of Science and Technology of China,
		School of Mathematical Sciences,
		No.96, JinZhai Road, Baohe District, Hefei,
		Anhui, 230026. P.R.China.}
	\email{wuhan90@mail.ustc.edu.cn}

	\author{Chang Lv}
	\address{State Key Laboratory of Information Security,
		Institute of Information Engineering,
		Chinese Academy of Sciences,
		Beijing 100093, P.R. China}
	\email{lvchang@amss.ac.cn}
			\thanks{H.W.\ was partially supported by NSFC Grant No. 12071448.  C.L. was partially supported by NSFC Grant No. 11701552.}
			
	\date{}




	\begin{abstract}
		For any number field,
		we prove that there exists a stacky curve of genus $1/2$  defined over the ring of its integers violating the local-global principle for integral points.
	\end{abstract}
	
	\maketitle

	
	\section{Introduction}
	
	\subsection{Background}
	Given a number field $K,$ let $\Ocal_K$ be the ring of its integers, and let $\Omega_K$ be the set of all its nontrivial places. Let $K_v$ be the completion of $K$ at $v\in \Omega_K.$ For $v$ is a finite place, let $\Ocal_v$ be the valuation ring of $K_v.$ For $v$ is an archimedean place, let $\Ocal_v=K_v.$
	Let $X$ be a finite type algebraic stack over $\Ocal_K.$ 	If the set $X(\Ocal_K)\neq\emptyset,$ then the set $X(\Ocal_v)\neq\emptyset$ for all $v\in \Omega_K.$  The converse does not
	always hold.  We say that $X$ violates the \defi{local-global principle for integral points} if $X(\Ocal_{v})\neq\emptyset$ for all $v\in\Omega_K,$ whereas $X(\Ocal_K)=\emptyset.$ For $K=\QQ,$ Darmo and Granvill \cite{DG95} implicitly gave an example of a stacky curve violating the local-global principle for integral points. In the paper \cite{BP20}, Bhargava and Poonen proved that for any stacky curve over $\Ocal_K$ of genus less than $1/2,$ it satisfies local-global principle for integral points. For $K=\QQ,$ they gave an example of a genus-$1/2$ stacky curve violating the local-global principle for integral points in loc. cit.

	Our goal is to generalize their counterexample to any number field. We will prove the following theorem.
	
	\begin{thm}[Theorem \ref{thm: main theorem}]
		For any number field $K,$  there exists a stacky curve of genus-$1/2$ over $\Ocal_K$ violating the local-global principle for integral points.
	\end{thm}
	
	The way to prove this theorem is to give an explicit construction of a genus-$1/2$ stacky curve violating the local-global principle for integral points. The paper is organised as follows. In Section \ref{section: notation and pre}, we set up the background by
	recalling some facts on stacky curves. Then we introduce a class of genus-$1/2$ stacky curves in Section \ref{section: class of stacky curves}. In Section \ref{section: existence of local points}, we prove that the stacky curves given in Section \ref{section: class of stacky curves} have local integral points.
	Finally, in Section \ref{section: main theorem}, we put some restrictions on the stacky curves given in Section \ref{section: class of stacky curves} so that they do not have integral points, then Theorem \ref{thm: main theorem} holds.

	\section{Notation and preliminaries}\label{section: notation and pre}
	
	\subsection{Notation}
	Given a number field $K,$ let $\Ocal_K$ be the ring of its integers, and let $\Omega_K$ be the set of all its nontrivial places. Let $\infty_K^r\subset \Omega_K$  be the subset of all real places. Let $K_v$ be the completion of $K$ at  $v\in \Omega_K.$ For $v$ is a finite place, let $\Ocal_v$ be the valuation ring of $K_v,$ and let $\FF_v$ be the residue field. For $v$ is an archimedean place, let $\Ocal_v=K_v.$
	We say that an element is a \defi{prime element}, if the ideal generated by this element is a prime ideal.
	If an element $p\in \Ocal_K$ is a prime element, we denote its associated valuation by $v_p,$ and its associated valuation ring (field) by $\Ocal_p$ (respectively $K_p$). Let $\overline{K}$ be an algebraic closure of $K.$

	\subsection{Stacky curves}
	In this subsection, we briefly recall some facts on stacky curves.  We refer to \cite{Ol16}, \cite{VZB19} and \cite{BP20} for more details.
	
	We say that $X$ is a \defi{stacky curve} over $K,$ if $X$ is a smooth, proper and geometrically
	connected $1$-dimensional Deligne-Mumford stack over $K$ that contains  a nonempty open substack isomorphic to a scheme, cf. \cite[Definiton 5.2.1]{VZB19}.
	Given a stacky curve $X$ over a number field $K,$ by \cite[Theorem 1.1]{KM97}, let $X_{\rm coarse}$ be its coarse moduli space, which is a smooth, projective and geometrically connected curve over $K.$ Let $\pi\colon X\to X_{\rm coarse}$  be the coarse space morphism. For any finite extension $L/K$ and any closed point  $P\in X_{\rm coarse}(L),$ let $G_P$ be the stabilizer of $X$ above $P,$ which is a finite group scheme over $K.$  Let $\calP \subset X_{\rm coarse}$ be the reduced finite subsheme above which the stabilizer is nontrivial. And
	$\pi$ is an isomorphism over the open subscheme $X_{\rm coarse}\backslash\calP.$
	Motivated by the Riemann-Hurwitz formula, the genus of $X$ is defined by
	\begin{equation} \label{eq:genus1}
		g(X)\colon=g(X_{\rm coarse})+\frac{1}{2}\sum_{P\in \calP} \left(1-\frac{1}{\deg G_P}\right)\deg P.
	\end{equation}
	This formula is stable under base field change. It can be defined using the geometrically closed points of $\calP$ by
	\begin{equation} \label{eq:genus2}
		g(X)\colon=g(X_{\rm coarse})+\frac{1}{2}\sum_{\overline{P}\in \calP(\overline{K})} \left(1-\frac{1}{\deg G_{\overline{P}}}\right).
	\end{equation}
	In particular, the genus is a nonnegative rational number. From this formula, the following lemma follows.
	
	\begin{lemma}(\cite[Lemma 6 and Proposition 8]{PV10})
		Let $X$ be a stacky curve over a  number  field $K,$ then $g(X)\geq 0.$ If $g(X)<1,$
		then $g(X_{\rm coarse})=0$ and $X$ is geometrically isomorphism to $\PP^1.$
	\end{lemma}
	
	It follows by the Hasse-Minkowski theorem that for a stacky curve of genus less than one over a  number  field, the local-global principle for rational points always holds. Bhargava and Poonen \cite[Theorem 5]{PV10} proved that the local-global principle for integral points always holds for a stacky curve of genus less than $1/2$ over a  number  field. Furthermore, Christensen \cite[Theorem 13.0.6]{Ch20} proved that it satisfies strong approximation.
	Because of these, we consider the  local-global principle for integral points of genus-$1/2$ stacky curves. We say that $\calX$ is a \defi{stacky curve} over $\Ocal_K,$ if $\calX$ is a proper algebraic stack over $\Ocal_K$ such that it is a stacky curve over $K,$ (i.e. $\calX_K$ its base change to $K,$ is a stacky curve).  For any $\Ocal_K$-algebra $R,$ let $\calX(R)$ be the set of isomorphism classes of $\Ocal_K$-morphisms $\Spec R\to \calX.$
	
	\section{A class of genus-$1/2$ stacky curves}\label{section: class of stacky curves}
	
	Let $K$ be a number field. Let $\mu_2\colonequals \Spec \Ocal_K[\lambda]/(\lambda^2-1)\subset \GG_m\colonequals \Spec \Ocal_K[\lambda,1/\lambda] $ be the closed subgroup scheme. Let $\ZZ/2\ZZ\colonequals \Spec \Ocal_K[\lambda]/(\lambda-1)\bigsqcup \Spec \Ocal_K[\lambda]/(\lambda+1).$ The following lemma states that  these two finite group schemes are isomorphic over $\Ocal_K[1/2].$
	
	\begin{lemma}\label{lemma: two finite group isom}
		Given a number field field $K,$ the natural morphism $\ZZ/2\ZZ\to \mu_2$ given by
		\begin{equation*}
			\Ocal_K[\lambda]/(\lambda^2-1) \to \Ocal_K[\lambda]/(\lambda-1)\times  \Ocal_K[\lambda]/(\lambda+1)
		\end{equation*}
		is a group homomorphism. And it is an isomorphism over $\Ocal_K[1/2].$
	\end{lemma}
	\begin{proof}
		By a direct check of group operators of these two group schemes, this is a group homomorphism. And the ring homomorphism base change to $\Ocal_K[1/2],$ is an isomorphism.  
	\end{proof}
	
	Let $p,~q$ be two coprime integers in $K.$
	Let $z^2-px^2-qy^2$ be a homogeneous polynomial in $\Ocal_K[x,y,z]$ with homogeneous coordinates $(x:y:z).$
	Let $\calY_{(p,q)} \colonequals \Proj \Ocal_K[x,y,z]/(z^2-px^2-qy^2),$ and let $Y_{(p,q)}$ be its base change to $K.$
	We define a $\mu_2$-action on $\calY_{(p,q)} $ by letting
	$\lambda \in \mu_2$ act as $(x:y:z) \mapsto (x:y:\lambda z)$.
	Let $[\calY_{(p,q)}/\mu_2]$ and $[Y_{(p,q)}/\mu_2]$ be the quotient stacks over $\Ocal_K$ and $K$ respectively.
	
	\begin{proposition}\label{prop: genus equals one half}
		The quotient stack $[\calY_{(p,q)}/\mu_2]$ is a Deligne-Mumford stack over $\Ocal_K[1/2].$ The quotient stack  $[Y_{(p,q)}/\mu_2]$ is a genus-$1/2$ stacky curve.
	\end{proposition}
	\begin{proof}
		Since $[\calY_{(p,q)}/(\ZZ/2\ZZ)]$ is a Deligne-Mumford stack over $\Ocal_K,$ the first argument follows from Lemma \ref{lemma: two finite group isom}. In particular, the quotient stack  $[Y_{(p,q)}/\mu_2]$ is a Deligne-Mumford stack. For a Deligne-Mumford stack, the properties of being smooth, proper and geometrically
		connected of dimension one follow from these properties of $Y_{(p,q)}.$ Let $\calP_{z=0}\subset Y_{(p,q)}$ be the finite $K$-subscheme defined by $z=0.$ The group $\mu_2$ acts freely on $\Proj K[x,y,z]/(z^2-px^2-qy^2)\backslash \calP_{z=0},$ so the stack $(\Proj K[x,y,z]/(z^2-px^2-qy^2)\backslash \calP_{z=0})/\mu_2$ is representable by a scheme, which is an open substack of $[Y_{(p,q)}/\mu_2].$ For $\Proj K[x,y,z]/(z^2-px^2-qy^2)\backslash \calP_{z=0}$ is geometrically isomorphic to $\GG_m,$  geometrically this action over it can be viewed as the action from the Kummer sequence $1\to \mu_2\to \GG_m\to \GG_m\to 1,$ hence the stack $(\Proj K[x,y,z]/(z^2-px^2-qy^2)\backslash \calP_{z=0})/\mu_2$ is geometrically isomorphism to $\GG_m.$
		So  $[Y_{(p,q)}/\mu_2]$ is a stacky curve and $g([Y_{(p,q)}/\mu_2]_{\rm coarse})=0.$ For $\mu_2$ acts trivially on $\calP_{z=0}$ containing two geometrically point, by the genus formula (\ref{eq:genus2}), we have $g([Y_{(p,q)}/\mu_2]) =1/2.$
	\end{proof}

	The stacky curves that we consider in this paper, are the quotient stacks of form $[Y_{(p,q)}/\mu_2].$ And we denote  $[\calY_{(p,q)}/\mu_2]$ by $\calX_{(p,q)}.$
	
	\section{Existence of local points}\label{section: existence of local points}
	
	In this section, we prove that the stacky curve $\calX_{(p,q)}$ has local integral points, i.e. the set $\calX_{(p,q)}(\Ocal_v)\neq\emptyset$ for all $v\in \Omega_K.$
	
	\begin{lemma}\label{lemma:local points}
		Given a number field $K,$ let $p,~q$ be two coprime integers in $K.$ Let $S=\infty_K^r\cup\{v\in \Omega_K^f | v(2pq)\neq 0\}$ be a finite set.  Then the set  $\calY_{(p,q)}(\Ocal_v)\neq\emptyset$ for all $v\in \Omega_K\backslash S.$
	\end{lemma}
	\begin{proof}
		For any  finite place $v\in \Omega_K,$ by Chevalley-Warning theorem (cf. \cite[Chapter I \S 2, Corollary 2]{Se73}), the set $\calY_{(p,q)}(\FF_v)\neq \emptyset.$ For any $v\in \Omega_K\backslash S,$ the scheme $\calY_{(p,q)}$ is smooth over $\Ocal_v.$  By the smooth lifting theorem, the set  $\calY_{(p,q)}(\Ocal_v)\neq\emptyset$ for all $v\in \Omega_K\backslash S.$
	\end{proof}
	\begin{remark}
		Consider the quotient morphism:  $\calY_{(p,q)}\to \calX_{(p,q)}.$ Then this lemma implies that the set $\calX_{(p,q)}(\Ocal_v)\neq\emptyset$ for all $v\in \Omega_K\backslash S.$
	\end{remark}
	In order to prove that the stacky curves $\calX_{(p,q)}$ has local integral points. We need to check that the set $\calX_{(p,q)}(\Ocal_v)\neq\emptyset$ for all $v\in S.$
	
	Let $\Ocal_K$-algebra $R$ be a principal ideal domain. We analysis the set $\calX_{(p,q)}(R)$ first.
	
	By definition of the quotient stack,
	a morphism $\Spec R \to \calX_{(p,q)}$
	is given by a $\mu_2$-torsor $T$ over $R$ equipped with a $\mu_2$-equivariant
	morphism $T \to \calY_{(p,q)}$.
	The torsors are classified by
	$\HH^1_{\textup{fppf}}(R,\mu_2)$,
	which is isomorphic to $R^\times/R^{\times 2}$,
	since $\HH^1_{\textup{fppf}}(R,\GG_m) = \Pic R = 0$.
	Explicitly, if $t \in R^\times$,
	the corresponding $\mu_2$-torsor is $T_t \colonequals \Spec R[u]/(u^2-t)$ and the $\mu_2$-action on $T_t$ is given by letting
	$\lambda \in \mu_2$ act as $u \mapsto \lambda u.$
	Let $\calY_{(p,q)t} \colonequals \Proj R[x,y,z']/(tz'^2-px^2-qy^2)$ be the twist of $\calY_{(p,q)}$ by $t.$ Consider the $\mu_2$-torsor $\calY_{(p,q)t}\times T_t$ over $\calY_{(p,q)t}.$ Define a morphism $\calY_{(p,q)t}\times T_t\to \calY_{(p,q)}$ given by
	\begin{align*}
		\Ocal_K [x,y,z]/(z^2-px^2-qy^2)& \to  R[x,y,z',u]/(tz'^2-px^2-qy^2, u^2-t)\\
		(x,y,z) & \mapsto (x,y,uz') .
	\end{align*}
	It is a  $\mu_2$-equivariant morphism. This gives a morphism
	$\pi_t \colon \calY_{(p,q)t} \to \calX_{(p,q)}$.
	To give a $\mu_2$-equivariant morphism $T_t \to \calY_{(p,q)}$
	is the same as giving a tripe $(a_1,a_2,a_3)\in R^3,$ and the $\mu_2$-equivariant morphism is given by
	\begin{align*}
		\Ocal_K [x,y,z]/(z^2-px^2-qy^2)& \to  R[u]/( u^2-t)\\
		(x,y,z) & \mapsto (a_1,a_2,a_3u) .
	\end{align*}
	And the tripe  $(a_1,a_2,a_3)$ gives a morphism $\Spec R \to \calY_{(p,q)t}$ defined by
	\begin{align*}
		R[x,y,z']/(tz'^2-px^2-qy^2)& \to R \\
		(x,y,z') & \mapsto (a_1,a_2,a_3) .
	\end{align*}
	Hence, to give a $\mu_2$-equivariant morphism $T_t \to \calY_{(p,q)}$
	is the same as giving a morphism $\Spec R \to \calY_{(p,q)t}.$
	Thus we obtain
	\begin{equation}\label{eq: local points decomposition}
		\calX_{(p,q)}(R) = \coprod_{t \in R^{\times}/R^{\times 2}} \pi_t(\calY_{(p,q)t}(R)).
	\end{equation}

	With this preparation, we have the following proposition.
	
	\begin{proposition}\label{prop: local points}
		Given a number field $K,$ let $p,~q$ be two coprime integers in $K.$  Then the set  $\calX_{(p,q)}(\Ocal_v)\neq\emptyset$ for all $v\in \Omega_K.$
	\end{proposition}
	\begin{proof}
		By Lemma \ref{lemma:local points}, we need to check that the set $\calX_{(p,q)}(\Ocal_v)\neq\emptyset$ for all $v\in S.$
		
		Suppose that $v\in \infty_K^r$ or  $v \nmid q.$ Then $q\in \Ocal_{v}.$ Since $qz^2-px^2-qy^2=0$ has a nontrivial solution $(x:y:z)=(0:1:1),$ we have $\calY_{(p,q)q}(\Ocal_v) \ne \emptyset.$ Hence, the set $\calX_{(p,q)}(\Ocal_v)\neq\emptyset.$
		
		Similarly, suppose that $v \nmid p,$ the sets $\calY_{(p,q)p}(\Ocal_v) \ne \emptyset$ and $\calX_{(p,q)}(\Ocal_v)\neq\emptyset.$
		
		Since $p,~q$ are two coprime integers,
		the set $\calX_{(p,q)}(\Ocal_v)\neq\emptyset$ for all $v\in \Omega_K.$
	\end{proof}
	
	\section{Genus $1/2$-stacky curves violating the local-global principle for integral points}\label{section: main theorem}
	
	Given a number field $K,$ we put some restrictions on the choice of integers $p,q$ 
	so that the stacky curve $\calX_{(p,q)}$ has no integral points, i.e. the set $\calX_{(p,q)}(\Ocal_K)=\emptyset.$ We choose $p,q$ in the following way.

	\subsection{Choosing prime elements}\label{subsection choosing elements}
	Given a number field $K,$ since the ideal class group of $K$ is finite, we take a positive integer $N$ such that $\Ocal_K[1/N]$ is a principal ideal domain. By Dirichlet's unit theorem, the group $\Ocal_K[1/N]^\times$ is a finitely generated abelian group. We assume that it is generated by $\{a_i\}$ for $i=1,\cdots,n.$
	By \v{C}ebotarev's density
	theorem and global class field theory applied to a ray class field, we can find a pair of two different odd prime elements $(p,q)$ such that
	\begin{enumerate}
		\item\label{assumption 3} $a_i\in K_p^{\times 2}$  for all $i=1,\cdots,n,$
		\item\label{assumption 4}  $q\notin K_p^{\times 2}.$
	\end{enumerate}
	We refer to \cite{Wu21} and \cite{Wu22} for more details. Then we have the following theorem.

	\begin{theorem}\label{thm: main theorem}
		Let $K$ be a number field. Let a positive integer $N$ and  a pair of two different odd prime elements  $(p,q)$ be  chosen as in Subsection \ref{subsection choosing elements}. Let $\calX_{(p,q)}$ be the stacky curve defined in Section \ref{section: class of stacky curves}.
		Then $\calX_{(p,q)}$ is a stacky curve of genus-$1/2$ over $\Ocal_K$ violating the local-global principle for integral points.
	\end{theorem}
	
	\begin{proof}
		By Proposition \ref{prop: genus equals one half}, the genus of $\calX_{(p,q)}$ is $1/2.$ By Proposition \ref{prop: local points}, the set $\calX_{(p,q)}(\Ocal_v) \ne \emptyset$ for any $v\in \Omega_K.$
		
		Next, we prove that the set  $\calX_{(p,q)}(\Ocal_K[1/N])=\emptyset.$
		For the ring $\Ocal_K[1/N]$ is a principal ideal domain,  in order to prove that $\calX_{(p,q)}(\Ocal_K[1/N]) = \emptyset,$ by the equality of sets (\ref{eq: local points decomposition}), it will be sufficient to prove that for any $t\in \Ocal_K[1/N]^{\times},$ the set $\calY_{(p,q)t}(\Ocal_K[1/N])=\emptyset.$  For $\Ocal_K[1/N]^{\times}$ is generated by $\{a_i\}$ for $i=1,\cdots,n,$ and by the chosen condition of Subsection \ref{subsection choosing elements} that $a_i\in K_p^{\times 2},$ we have $ \calY_{(p,q)t}$ is isomorphic to $\calY_{(p,q)}$ over $K_p$ for all  $t\in \Ocal_K[1/N]^{\times}.$
		By the choice of elements $q,$ the set $\calY_{(p,q)}(K_p)=\emptyset.$ So $\calX_{(p,q)}(\Ocal_K[1/N]) = \emptyset,$ which implies that $\calX_{(p,q)}(\Ocal_K)=\emptyset.$
		
		So the stacky curve $\calX_{(p,q)}$ is  of genus-$1/2$ and violating the local-global principle for integral points.
	\end{proof}
	
	\begin{remark}
		This theorem implies that
		the chosen stacky curve $\calX_{(p,q)}$ violates strong approximation in the sense of \cite{Ch20}.
	\end{remark}


		\begin{footnotesize}
			\noindent\textbf{Acknowledgements.} The authors would like to thank D.S. Wei and W.Z. Zheng for many fruitful discussions. 
		\end{footnotesize}


		\end{document}